\documentclass{amsart}
\usepackage{graphicx}
\usepackage{amsmath}
\usepackage{amsfonts}
\usepackage{amssymb}

\newtheorem{thm}{Theorem}
\newtheorem{cor}[thm]{Corollary}
\newtheorem{lem}[thm]{Lemma}
\newtheorem{prop}[thm]{Proposition}
\theoremstyle{definition}

\newtheorem*{theorem*}{Theorem}

\newcommand{\la}{\left\langle}
\newcommand{\ar}{\right\rangle}
\newcommand{\NN}{\mathbb{N}}
\newcommand{\ZZ}{\mathbb{Z}}
\newcommand{\RR}{\mathbb{R}}
\newcommand{\CC}{\mathbb{C}}

\newcommand{\nd}{\!\!\not|}
\newcommand{\fii}{\varphi}

\theoremstyle{remark}

\numberwithin{equation}{section}
\newcommand{\norm}[1]{\left\Vert#1\right\Vert}
\newcommand{\abs}[1]{\left\vert#1\right\vert}

\def\<{\langle}
\def\>{\rangle}
\begin{document}
\title[]{Orthogonality questions in the Hardy space related to $\zeta$-zeros}
\author{Francisco Calderaro, Juan Manzur, Waleed Noor \and Charles Santos}%
\address{IMECC, Universidade Estadual de Campinas, Campinas-SP, Brazil.}
\email{$\mathrm{francisco.calderaro027@gmail.com}$ (Francisco Calderaro)}

\email{$\mathrm{waleed@unicamp.br}$ (Waleed Noor) Corresponding Author}
\address{ICMC, University of São Paulo, São Carlos-SP, Brazil.}
\email{$\mathrm{ch.charlesfsantos@gmail.com}$ (Charles F. Santos)}

\address{Departamento de Matemáticas y Estadística, Universidad del Norte, Barranquilla, Colombia.}
\email{$\mathrm{juan\_manzur123@hotmail.com}$ (Juan Manzur)}

\begin{abstract} A Hardy space approach to the Nyman-Beurling and B\'aez-Duarte criterion for the Riemann Hypothesis (RH) was introduced recently in \cite{Noor} and further developed in \cite{Aditya preprint}. It states that the RH holds if and only if a particular sequence of functions $(h_k)_{k\geq 2}$ is complete in the Hardy space $H^2$. This article is concerned with orthogonality questions related to the family $(h_k)_{k\geq 2}$. The first goal is to analyze the orthogonal complement of $\mathcal{N}=\mathrm{span}(h_k)_{k\geq 2}$ in $H^2$. Unbounded Toeplitz operators on $H^p$ spaces and de Branges-Rovnyak spaces play a central role and our results show that the size and dimension of $\mathcal{N}^\perp$ reveal information on the zeros of the Riemann $\zeta$-function. The second goal is to show that $(h_k)_{k\geq 2}$ possesses a complete biorthogonal sequence in $H^2$. We also discuss a folklore conjecture about the number of $\zeta$-zeros if the RH fails.
\end{abstract}

{\subjclass[2010]{Primary; Secondary}}
\maketitle{}

\section*{Introduction}
A classical result of Nyman and Beurling (see \cite{Beurling RH}, \cite{Nyman}) shows that the Riemann Hypothesis (RH) is equivalent to the completeness of \(\{\rho_\lambda : 0\leq \lambda \leq 1\}\) in \(L^2(0,1)\)  where \(\rho_\lambda(x) = \{\lambda / x\} - \lambda \{1/x\} \) and \(\{x\}\) denotes the fractional part of \(x\). This is furthermore equivalent to the characteristic function \(\chi_{(0,1)}\) belonging to the closed linear span of \(\{\rho_\lambda : 0\leq \lambda \leq 1\}\). Half a century later Báez-Duarte \cite{Baez-Duarte} strenghthened this result by showing that RH is true if and only if \(\chi_{(0,1)} \in \overline{\mathrm{span}\{\rho_{1/k}: k \geq 2\}}\). See the expository article of Bagchi \cite{Bagchi} and the survey of Balazard \cite{balazard survey}. Recently the Nyman-Beurling and B\'aez-Duarte approaches to the RH have been explored via tools from Hardy space $H^2$ theory \cite{Noor} and other analytic function spaces \cite{Aditya preprint}.

For each $k\geq2$, define 
\begin{align*}
h_{k}(z)=\dfrac{1}{1-z}\log\left(\dfrac{1+z+\dots+z^{k-1}}{k}\right)
\end{align*}
and denote by $\mathcal{N}$ the linear span of $\{h_k:k\geq 2\}$. That each $h_k$ belongs to $H^2$ was proved in \cite[Lemma 7]{Noor}. One of the main results of \cite{Noor} was a reformulation of B\'aez-Duarte's result as a completeness problem in $H^2$. Then in \cite{Aditya preprint} the same completeness problem in the $H^p$ spaces was shown to provide zero-free half planes for the Riemann $\zeta$-function. We state both results here as one.

\begin{thm}\label{Noor density result}
The RH holds if and only if $\mathcal{N}$ is dense in $H^2$. If $\mathcal{N}$ is dense in $H^p$ for some $1<p\leq 2$, then 
\[
\zeta(s)\neq 0 \ \ \mathrm{for} \ \ \Re(s)>\frac{1}{p}.
\]
The density of $\mathcal{N}$ in $H^1$ gives the known zero-free region $\Re(s)\geq1$.		
	
\end{thm}

Although $\mathcal{N}$ is unconditionally dense in $H^p$ for $0<p<1$ (see \cite[Cor. 4.6]{Aditya preprint}), this provides no new information regarding $\zeta$-zeros. Also note that the RH is equivalent to $\mathcal{N}^\perp=\{0\}$ in $H^2$ by Theorem \ref{Noor density result}. This suggests the question of the size and dimension of $\mathcal{N}^\perp$ and their possible relation to the $\zeta$-zeros. The first archetypal result addressing this is the following (See \cite[Thm. 12]{Noor}). 

\begin{thm}\label{Nperp local D}
$
\mathcal{N}^\perp\cap \mathcal{D}_{{\delta}_1}=\{0\},
$
where $\mathcal{D}_{{\delta}_1}$ denotes the local Dirichlet space at $1$.
		\end{thm}
Since $\mathcal{D}_{{\delta}_1}\subset H^2\subset H^p$ for $0<p<2$, Theorems $\ref{Noor density result}$ and $\ref{Nperp local D}$ inspire the following question: which linear spaces of analytic functions on $\mathbb{D}$ with $X_1\subset H^2\subset X_2$ satisfy
\[
(1) \  \mathcal{N}^\perp\cap X_1=\{0\} \ \ \ \mathrm{and} \ \ \ (2) \  \mathcal{N} \ \mathrm{is}  \ \mathrm{dense} \ \mathrm{in} \ X_2 \ ?
\]
Interestingly, both $(1)$ and $(2)$ are equivalent if one takes 
$X_2=X$ a Frechet space and $X_1=X^*$ its Cauchy dual (see Proposition \ref{density <-> orthogonality}). In Section 2 we analyze this \emph{orthogonality question} using tools from de Branges-Rovnyak spaces and unbounded Toeplitz operators on $H^p$ (see Proposition \ref{hk/phi Hp}), and in particular develop a new approach for finding zero-free half-planes for the $\zeta$ function (see Theorem \ref{phiN dense zero free halfplane}).

In Section 3 we deal with the \emph{biorthogonality question}: Does the sequence $(h_k)_{k\geq 2}$ posses a complete biorthogonal sequence in $H^2$? (see Subsection $1.5$). We affirmatively answer this question (see Theorem \ref{mainthm}). To provide some context, Vasyunin showed in \cite{vasyunin} that the B\'aez-Duarte sequence $\{\rho_{1/k}: k \geq 2\}$ is minimal by constructing a biorthogonal sequence for it in $L^2(0,1)$. Whereas the sequence $\{\rho_{1/k}: k \geq 2\}$ is not complete in $L^2(0,1)$, the completeness of $(h_k)_{k\geq 2}$ in $H^2$ is equivalent to the RH by Theorem \ref{Noor density result}. The property of completeness of a sequence is not in general inherited by its biorthogonal sequence. But it may do so in very special cases such as for sequences of complex exponentials in $L^2(-\pi,\pi)$ (see \cite{young}). Therefore an affirmative answer to this question (independent of RH) is intriguing. 

Finally in  Section 4 we discuss a folklore conjecture about the $\zeta$-zeros which we name the \emph{RH failure} (RHF) conjecture: \emph{If the RH fails, then it fails infinitely often}. More precisely, either $\zeta$ has no nontrivial zeros outside the critical line, or it has infinitely many. The main result of this section (Theorem \ref{RHF implies dim N perp 0 or infty}) states that
\[
\mathrm{RHF} \ \mathrm{conjecture} \implies \mathrm{dim}(\mathcal{N}^\perp) \ \mathrm{is} \ \mathrm{either} \ 0 \ \mathrm{or} \ \infty.
\] 
We were unable to locate a reference for this conjecture in the literature. But an informative discussion of this conjecture does appear in mathoverflow.net \cite{Feldman}.

\section{Preliminaries} \label{Sec:Background}

We denote by $\mathbb{D}$ and $\mathbb{T}$ the open unit disk and the unit circle respectively. By $\mathrm{Hol}(\mathbb{D})$ we denote the space of holomorphic functions on $\mathbb{D}$ and the shift operator is defined by $(Sf)(z)=zf(z)$ for $f\in \mathrm{Hol}(\mathbb{D})$.

\subsection{Hardy spaces and the Smirnov class} A holomorphic function $f$ on $\mathbb{D}$ belongs to the Hardy-Hilbert space $H^2$ if
\[
||f||_{H^2}=\sup_{0\leq r<1}\left(\frac{1}{2\pi}\int_0^{2\pi}|f(re^{i\theta})|^2d\theta\right)^{1/2}<\infty.
\]
The space $H^2$ is a Hilbert space with the $\ell^2$-inner product
\[
\langle f,g \rangle=\sum_{n=0}^\infty a_n\overline{b_n},
\]
where $(a_n)_{n\in\mathbb{N}}$ and $(b_n)_{n\in\mathbb{N}}$ are the Fourier coefficients for $f$ and $g$ respectively. For any $f\in H^2$ and $\zeta\in\mathbb{T}$, the radial limit $f^*(\zeta):=\lim_{r\to 1^-}f(r\zeta)$ exists $m$-a.e. on $\mathbb{T}$, where $m$ denotes the normalized Lebesgue measure on $\mathbb{T}$. Analogously for $p>0$, the Hardy space $H^p$ consists of those holomorphic $f$ on $\mathbb{D}$ such that
\[
\Vert f \Vert_{H^p}^p = \sup_{0<r<1} \int_\mathbb{T} \abs{f(rz)}^p \,\mathrm{d}m(z) \,<\, \infty .
\]
The $H^p$ are Banach spaces for $p\geq 1$ and complete metric spaces for $0<p<1$ and $H^\infty$ denotes the space of bounded holomorphic functions on $\mathbb{D}$. A function $f$ is called a cyclic vector for the shift $S$ in $H^p$ if $\mathrm{span}(S^nf)_{n\geq 0}=\mathbb{C}[z]f$ is dense in $H^p$. When $p=2$ these cyclic vectors are commonly known as outer functions. Since $H^\infty\subset H^p$ for all $p>0$, the $H^\infty$ outer functions are cyclic for all $H^p$ spaces. The Smirnov class $N^+$ consists of all holomorphic functions $g/h$ on $\mathbb{D}$ such that $g,h\in H^\infty$ and $h$ is an outer function. The space $N^+$ is a topological algebra with respect to pointwise multiplication and $g/h\in N^+$ is a unit if both $g$ and $h$ are outer functions.
The topology on $N^+$ can be metrized with the translation-invariant complete metric
\[ \label{Eq:Ndistance}
d(f,g) = \int_\mathbb{T} \log(1+\abs{f-g}) \,\mathrm{d}m, \qquad f,g \in N^+ .
\]
Similar to $H^p$ spaces, convergence in $N^+$ implies locally unifom convergence on $\mathbb{D}$ and functions have radial limits $m$-a.e. on $\mathbb{T}$. In fact we have $H^p\subset H^q \subset N^+$ for all $0 < q< p \leq \infty$. Duren \cite{Duren} is a classic reference for the $H^p$ and $N^+$ spaces.

\subsection{Local Dirichlet spaces} Let $\mu$ be a finite positive Borel measure on $\mathbb{T}$, and let $P\mu$ denote its Poisson integral. The \emph{generalized Dirichlet space} $\mathcal{D}_\mu$ consists of $f\in H^2$ satisfying
\[
\mathcal{D}_\mu(f):=\int_\mathbb{D}\abs{f'(z)}^2P\mu(z)dA(z)<\infty.
\]
Then $\mathcal{D}_\mu$ is a Hilbert space with norm $\norm{f}_{\mathcal{D}_\mu}^2:=\norm{f}_2^2+\mathcal{D}_\mu(f)$. If $\mu=m$, then $\mathcal{D}_m$ is the classicial Dirichlet space. If $\mu=\delta_\zeta$ is the Dirac measure at $\zeta\in\mathbb{T}$, then $\mathcal{D}_\zeta:=\mathcal{D}_{\delta_\zeta}$ is called the \emph{local Dirichlet space} at $\zeta$ and in particular
\begin{equation}\label{Dirichlet integral}
\mathcal{D}_{\delta_\zeta}(f)=\int_\mathbb{D}\abs{f'(z)}^2\frac{1-\abs{z^2}}{\abs{z-\zeta}^2}dA(z).
\end{equation}
The recent book \cite{Dirichlet book} contains a comprehensive treatment of local Dirichlet spaces and the following result establishes a criterion for their membership.
\begin{thm}\label{Local Dirichlet membership}(See  \cite[Thm. 7.2.1]{Dirichlet book})
	Let $\zeta\in\mathbb{T}$ and $f\in\mathrm{Hol}(\mathbb{D})$. Then
	$\mathcal{D}_{\delta_\zeta}(f)<\infty$ if and only if \[f(z)=(z-\zeta)g(z)+a\] for some $g\in H^2$ and $a\in \mathbb{C}$. In particular $f^*(\zeta)$ exists for all $f\in\mathcal{D}_{\zeta}$.
	\end{thm}
So each local Dirichlet space $\mathcal{D}_{\zeta}=(S-\zeta I)H^2+\mathbb{C}$ is a proper subspace of $H^2$. We define the $H^p$-analogues of these spaces for $p>0$ by
\[
\mathcal{D}_\zeta^p:=(S-\zeta I)H^p+\mathbb{C} \ 
\]
 and note that $\mathcal{D}_{\zeta}^2=\mathcal{D}_{\delta_\zeta}$ and $\mathcal{D}_{\zeta}^p\subsetneq \mathcal{D}_{\zeta}^q$ for $q<p$ since $H^p\subsetneq H^q$. Straightforward but lengthy computations show that $\mathcal{D}_{\zeta}^p\subsetneq H^2$ for $p>1$ and $H^2\subsetneq \mathcal{D}_{\zeta}^p$ for $0<p<\frac{2}{3}$.

\subsection{The de Branges-Rovnyak spaces}
Given $\psi\in L^\infty(\mathbb{T})$, the corresponding Toeplitz operator $T_\psi:H^2\to H^2$ is defined by
\[
T_\psi f:=P_+(\psi f )
\] 
where $P_+ : L^2(\mathbb{T})\to H^2$ denotes the orthogonal projection of $L^2(\mathbb{T})$ onto $H^2$. 
Clearly $T_\psi$ is a bounded operator on $H^2$ with $||T_\psi|| \leq||\psi||_{L^\infty}$. If $h\in H^\infty$, then
$T_h$ is the operator of multiplication by $h$ and its adjoint is $T_{\overline{h}}$. Given $b$ in the closed unit ball of $H^\infty$, the \emph{de Branges-Rovnyak}
space $\mathcal{H}(b)$ is the image of $H^2$ under the operator $(I -T_bT_{\overline{b}})^{1/2}$.
The general theory of $\mathcal{H}(b)$ spaces divides into two distinct cases, according to whether $b$ is an extreme point or a non-extreme point of the unit ball of $H^\infty$. We shall be concerned only with the non-extreme case. In this case there exists a unique outer function $a\in H^\infty$ such that $a(0)>0$ and $|a^*|^2+|b^*|^2=1$ a.e. on $\mathbb{T}$. The pair $(b,a)$ is called a \emph{Pythagorean pair} and the function $b/a$ belongs to the Smirnov class $N^+$. That all $N^+$ functions arise as the quotient of a pair associated to a non-extreme function was shown by Sarason \cite{unbounded toeplitz}.  The two-volume work (\cite{Hb book vol 1}\cite{Hb book vol 2}) is an encyclopedic reference for these spaces. 

If $\varphi$ is a rational function in $N^+$ the corresponding pair $(b,a)$ is also rational (see \cite[Remark. 3.2]{unbounded toeplitz}). Constara and Ransford \cite{Costara-Ransford} characterized the rational pairs $(b,a)$ for which $\mathcal{H}(b)$ is a generalized Dirichlet space.

\begin{thm}\label{which DR=Dirichlet}(See \cite[Theorem 4.1]{Costara-Ransford}) Let $(b,a)$ be a rational pair and $\mu$ a finite positive measure on $\mathbb{T}$. Then $\mathcal{H}(b)=\mathcal{D}_\mu$ if and only if
	\begin{enumerate}
		\item the zeros of $a$ on $\mathbb{T}$ are all simple, and
		\item the support of $\mu$ is exactly equal to this set of zeros.
	\end{enumerate}
	\end{thm}

As an example, if $(b,a)$ is the rational pair associated with the $N^+$ function $\varphi(z)=\frac{1}{1-\zeta}$ for $\zeta\in\mathbb{T}$, then $\mathcal{H}(b)=\mathcal{D}_\zeta$ is a local Dirichlet space.

\subsection{Unbounded Toeplitz operators on $H^p$}  Sarason \cite{unbounded toeplitz} demonstrated how $\mathcal{H}(b)$ spaces appear naturally as the domains of some unbounded Toeplitz operators. Let $\varphi$ be holomorphic in $\mathbb{D}$ and $T_\varphi$ the operator of multiplication by $\varphi$ on the domain 
\begin{equation}\label{Toeplitz domain}
\mathrm{dom}(T_\varphi)=\{f\in H^2:\varphi f\in H^2\}.
\end{equation}
Then $T_\varphi$ is a closed operator, and $\mathrm{dom}(T_\varphi)$ is dense in $H^2$ if and only if $\varphi\in N^+$ (see \cite[Lemma 5.2]{unbounded toeplitz}). In this case its adjoint $T_\varphi^*$ is also densely defined and closed. In fact the domain of $T_\varphi^*$ is a de Branges-Rovnyak space. 

\begin{thm}\label{dom of adjoint}(See \cite[Prop. 5.4]{unbounded toeplitz})
Let $\varphi$ be a nonzero function in $N^+$ with $\varphi=b/a$, where $(b,a)$ is the associated pair. Then $\mathrm{dom}(T_\varphi^*)=\mathcal{H}(b)$.
\end{thm}
Choosing the symbol $\varphi(z)=\frac{1}{\zeta-z}$ in Theorem \ref{dom of adjoint} in conjunction with Theorem \ref{which DR=Dirichlet} gives $\mathrm{dom}(T_\varphi^*)=\mathcal{D}_\zeta$ which played a key role in the proof of Theorem \ref{Nperp local D} (see \cite{Noor}). Our goal here is to extend these ideas to $H^p$ spaces for all $p > 1$. Let $\varphi \in N^+$ and define the analytic Toeplitz operator on $H^p$ with symbol $\varphi$ by
\[
T_\varphi f = \varphi f, \ \ \mathrm{where} \quad f \in \mathrm{dom}_p\,(T_\varphi):=\{f\in H^p \,:\, \phi  f \in H^p\}.
\]
These $T_\varphi$ are bounded on $H^p$ precisely when $\varphi\in H^\infty$ (see the survey article \cite{Vukotic}). For $\varphi=\frac{b}{a}\in N^+$ with $a,b\in H^\infty$ and $a$ outer as usual, these $T_\varphi$ are densely defined on $H^p$ for $p>1$. Indeed, $\mathrm{dom}_p(T_\varphi)$ contains the dense subspace $aH^p$ since $a$ is outer and $T_\varphi (aH^p)= bH^p\subset H^p$. It follows then that the adjoint $T_\varphi^*$ is well-defined on the dual $(H^p)^*=H^q$ where $\frac{1}{p}+\frac{1}{q}=1$. The domain of $T_\varphi^*$ is then defined  by
\[
\mathrm{dom}_q(T_\varphi^*):=\{g\in H^q:\exists \ h\in H^q \ s.t \ \langle f,h\rangle=\langle \varphi f, g\rangle \ \ \forall \ f\in \mathrm{dom}_p\,(T_\varphi)
\}
\]
where $\langle f,h\rangle:=\int_{\mathbb{T}}f\overline{h}dm$ represents the $H^p$-$H^q$ duality. The elements in $\mathrm{dom}_q(T_\varphi^*)$ can be characterized  via the bounded Toeplitz operators $T_a$ and $T_b$ as follows.

\begin{lem}\label{dom Toeplitz Hp}
Given $\varphi=\frac{b}{a}\in N^+$ as described above, a function $g\in \mathrm{dom}_q(T_\varphi^*)$ if and only if there exists an $h\in H^q$ such that $T_b^* g = T_a^* h$.
\end{lem}
\begin{proof}
Suppose $g\in\mathrm{dom}_q(T_\varphi^*)$. Then  $\langle f,h\rangle=\langle \varphi f, g\rangle$ for some $h\in H^q$ and for all $f\in aH^p\subset \mathrm{dom}_p\,(T_\varphi)$. Writing $f=a\Tilde{f}$ for $\Tilde{f}\in H^p$, we get
\[
\langle f,h\rangle=\langle \varphi f, g\rangle \ \Leftrightarrow \ \langle f,h\rangle=\langle bf/a, g\rangle \ \Leftrightarrow \ \langle a\Tilde{f},h\rangle=\langle b\Tilde{f}, g\rangle \ \forall \ \tilde{f}\in H^p 
\]
which is equivalent to $T^*_ah=T^*_bg$. This argument works in both directions because $a$ is an $H^\infty$ outer function and hence $a H^p$ is dense in $H^p$.
\end{proof}

We can now extend the identity $\mathrm{dom}(T_\varphi^*)=\mathcal{D}_\zeta$ from $H^2$ to all $H^p$ with $p>1$.
\begin{prop} \label{dom D_zeta^p} Let $\varphi(z)=\frac{1}{\zeta-z}$ for $\zeta\in \mathbb{T}$ and $T_\varphi$ the densely defined Toeplitz operator on $H^p$ for any $p>1$. We then have
\[
\mathrm{dom}_q ( T_\varphi^*) = \mathcal{D}_\zeta^q, \ \ \mathrm{where} \ \ \frac{1}{p}+\frac{1}{q} = 1.
\]
\end{prop}
\begin{proof}
Choosing $\varphi = \frac{b}{a}$ with $b(z) = 1$ and $a(z) = \zeta-z$ (which is outer) in Lemma \ref{dom Toeplitz Hp}, we get $g \in \mathrm{dom}_q(T_\varphi^*)$ if and only if $g = T_{\zeta-z}^*h = (\overline{\zeta}I - S^*)h$ for some $h \in H^q$. Therefore, it suffices to verify that $(\overline{\zeta}I - S^*) H^q = \mathcal{D}_\zeta^q$. For any $h\in H^q$, we have 
\[
(\overline{\zeta}I - S^*)h = (\zeta I - S) (-\overline{\zeta}S^*h) + \overline{\zeta}h(0) \in  \mathcal{D}_\zeta^q:=(S-\zeta I)H^q+\mathbb{C}
\]
and therefore $(\overline{\zeta}I - S^*) H^q \subset \mathcal{D}_\zeta^q$. Conversely, if $h\in H^q$ and $c\in \mathbb{C}$ then
\[(\zeta I - S) h + c = (\overline{\zeta}I - S^*) \zeta (c-Sh) \in (\overline{\zeta}I - S^*) H^q\]
and hence $\mathcal{D}_\zeta^q\subset (\overline{\zeta}I - S^*) H^q$ which concludes the proof.
\end{proof}

\subsection{Cauchy duality} \label{Subsec:Cauchy} Let $X$ be a complete metrizable linear subspace of $\mathrm{Hol}(\mathbb{D})$. Inspired by terminology used by Malman and Seco \cite{Malman-Seco}, we call $X^*$ the Cauchy dual of $X$ if any continuous linear functional on $X$ can be represented by the Cauchy pairing
\[
\< f , g \> := \lim_{r\to1^-} \int_{\mathbb{T}} f(r\zeta) \overline{g(r\zeta)} \,\mathrm{d}m(\zeta) \,,\quad f\in X, \ g\in X^* \,.
\]
If $H^2 \subset X$ then $X^* \subset H^2$ and vice-versa. Hence when both $f$ and $g$ are in $H^2$, the pairing above reduces to the standard inner product in $H^2$. Some examples of Cauchy duals for our context are listed below (see \cite{Duren},\cite{Girela},\cite{Yanagihara}).

\begin{enumerate}
\item $H^p$ and $H^q$ for $p>1$ and $1/p + 1/q = 1$,
\item $H^1$ and $\mathrm{BMOA}$ (analytic functions with bounded mean oscillation on $\mathbb{T}$),
\item $H^p$ for $1/2<p<1$ and $\Lambda_\alpha$ (the Lipschitz class of Hol$(\mathbb{D})$-functions with $\alpha$-H\"older continuous extension to $\mathbb{T}$, where $\alpha = 1/p-1$).
\item $N^+$ and the \emph{Gevrey class} $\mathcal{G}$ (Hol$(\mathbb{D})$-functions whose Taylor coefficients satisfy
$
a_n = O(e^{-c\sqrt{n}})
$
for some constant $c>0$).
\end{enumerate}

A deep result of Davis and McCarthy \cite{Davis-McCarthy} shows that the class $\mathcal{G}$ coincides with the universal multipliers for all non-extreme $\mathcal{H}(b)$ spaces. In particular $\mathcal{G}\subset \mathcal{H}(b)$ for all non-extreme $b$. The concept of Cauchy duality leads to an equivalence between orthogonality and density questions involving $\mathcal{N}$ which is explored in Section \ref{Sec:Nperp}.

\subsection{Minimality and biorthogonality}
Let \(\mathcal{H}\) be a Hilbert space. Two sequences \((e_n)_{ n \in \NN}\) and \((f_n)_{ n \in \NN}\) in \(\mathcal{H}\) are said to be \emph{biorthogonal} to each other if
\[\la e_n, f_m \ar = \delta_{nm} \ \ \ \forall \ n,m \in \NN\]
where \(\delta_{nm}\) is the Kronecker delta. The sequence (\(e_n)_{ n \in \NN}\) is called $\emph{minimal}$ if $e_n\notin\overline{\mathrm{span}}(e_k)_{k\neq n}$ for all $n\in\mathbb{N}$. The notions of biorthogonality, minimality and completeness are all related via the following well-known result.
\begin{prop} \label{biorthogonality}(see \cite[Lemma 3.3.1]{christensen})
	Let \((e_n)_{ n \in \NN}\) be a sequence in \(\mathcal{H}\). Then,
	\begin{itemize}
		\item[(i)] \((e_n)_{ n \in \NN}\) has a biorthogonal sequence if and only if \((e_n)_{ n \in \NN}\) is minimal.
		\item[(ii)] If \((e_n)_{ n \in \NN}\) has a biorthogonal sequence, then \((e_n)_{ n \in \NN}\) is complete in \(\mathcal{H}\) if and only if its biorthogonal sequence is unique.
	\end{itemize}
\end{prop}
In Section 3 we shall prove that the sequence  $(u_k)_{k \geq 2}$ defined by
\[
u_k(z) = \sum_{d | k} \frac{\mu(k/d)}{k/d} (z^{d-1} - z^d)
\]
forms a complete biorthogonal sequence for $(h_k)_{k\geq 2}$ in $H^2$, where \(\mu\) denotes the Möbius function defined on $\mathbb{N}$ by $\mu(k) = (-1)^s$ if $k$ is the product of
$s$ distinct primes, and $\mu(k) = 0$ otherwise.  



\subsection{The zeta kernels}
Let $X\subset\mathrm{Hol}(\mathbb{D})$ be a topological vector space where the monomials $(z^k)_{k\in\mathbb{N}}$ form a Schauder basis. It was shown in \cite{Aditya preprint} that for each $s\in\mathbb{C}\setminus\{0\}$ a linear functional $\Lambda^{(s)}$ can be defined on $X$ by assigning
\[
\Lambda^{(s)}(1) = -\frac{1}{s},\quad \Lambda^{(s)}(z^k) = -\frac{1}{s}\left((k+1)^{1-s}-k^{1-s}\right) \quad (k\geq1).
\] 
In particular $\Lambda^{(s)}$ is bounded on $H^p$ for $1< p\leq 2$ if $\Re s > 1/p$ and on $H^1$ if $\Re s \geq 1$ (see \cite[Prop. 4.7]{Aditya preprint}). So there exist functions $\kappa_s\in H^q$ with $1/p+1/q=1$ such that $\Lambda^{(s)}(f)=\langle f,\kappa_s\rangle$. The function $\kappa_s$ will be called the zeta kernel at $s$ and 
\begin{equation}\label{kappa_s}
\kappa_s(z) = \sum_{k=0}^\infty \phi_k(\bar{s}) z^k \ \ \mathrm{where} \ \ \phi_k(s) := \Lambda^{(s)}(z^k).
\end{equation}
The name comes from their relation to $h_k$ and $\zeta$ via the important identity
\begin{equation} \label{Zeta Kernel Relation}
\Lambda^{(s)}(h_k) = \langle h_k,\kappa_s\rangle= -\frac{\zeta(s)}{s} (k^{1-s}-1) \ \ \forall \ \ \Re{s}>1/2, \ k\geq 2.
\end{equation}
 The identity \eqref{Zeta Kernel Relation} appears in \cite{Aditya preprint} but we provide an alternate proof in the appendix for the sake of completeness. It is important to mention that the definition of $h_k$ in \cite{Aditya preprint} has an additional factor of $1/k$ which has been adjusted in \eqref{Zeta Kernel Relation} accordingly. The zeta kernels are used in Chapters 3 and play a key role in Chapter 4.

\section{The orthogonality question} \label{Sec:Nperp}

The objective of this section is to develop a framework for proving when
\begin{equation}\label{Nperp intesection D}
\mathcal{N}^{\perp}\cap L=\lbrace 0\rbrace
\end{equation}
for topological vector spaces $L\subset H^2$. Since $\mathcal{N}^\perp=\{0\}$ is equivalent to the RH by Theorem \ref{Noor density result}, one may also ask if solutions to \eqref{Nperp intesection D} can lead to new zero-free half-planes for the $\zeta$-function. We start by showing that Cauchy duality serves as a bridge between this orthogonality question and completeness questions. 

\begin{prop} \label{density <-> orthogonality}
Let $X$ be a topological linear space with $H^2 \subset X \subset \mathrm{Hol}(\mathbb{D})$, where the inclusions are continuous. If $\mathcal{N}$ is dense in $X$, then
\[
\mathcal{N}^\perp \cap X^* = \{0\} 
\]
 where $X^*\subset H^2$ is the Cauchy dual of $X$. The converse holds when $X$ is Fr\'echet.
\end{prop}
\begin{proof}
First note that since both $\mathcal{N}^\perp$ and $X^*$ are subspaces of $H^2$, their intersection above makes sense. Let $\< f , g \>$ denote the Cauchy pairing for $f\in X$ and $g\in X^*$ and recall that this pairing becomes the usual $H^2$-inner product $\< f , g \>_{H^2}$ when $f,g\in H^2$ (see Subsection 1.5). Therefore if $\mathcal{N}$ is dense in $X$, then
\[
g\in \mathcal{N}^\perp \cap X^*
\implies \< f , g \> = \< f , g \>_{H^2}=0 \  \ \mathrm{for} \ \mathrm{all} \ \ f\in \mathcal{N}
\]
which implies that $g$ must be identically zero in $X^*$. Conversely if $X$ is additionally a Fr\'echet space, then we have access to the Hahn-Banach Theorem. Indeed if $\mathcal{N}$ is not dense in $X$, then there exists a non-zero $g\in X^*$ such that $\< f , g \> = 0 \ \forall\, f\in \mathcal{N}$. This implies that $g$ is $H^2$-orthogonal to $\mathcal{N}$ and hence $g\in\mathcal{N}^\perp \cap X^* \neq \{0\}$.
\end{proof}

 Since $\mathcal{N}$ is dense in $H^p$ for $0<p<1$ (see \cite[Cor. 4.6]{Aditya preprint}), it is also dense in $N^+$ since $H^p\subset N^+$ for all $p>0$. Therefore it follows by Proposition \ref{density <-> orthogonality} that $\mathcal{N}^{\perp}\cap L=\lbrace 0\rbrace$ if $L$ is the Lipschitz class $\Lambda_\alpha$ ($1/2<p<1$ and $\alpha=1/p-1$) or the Gevrey class $\mathcal{G}$ (see Subsection 1.5). However to obtain new zero-free half-planes for $\zeta$, we need $L\subset H^2$ to be large enough to contain some $H^q$ space for $q\geq 2$.

\begin{cor}\label{Nperp interects H^p} If $\mathcal{N}^{\perp}\cap H^q=\lbrace 0\rbrace$ for some $q\geq 2$, then $\zeta(s)\neq 0$ for $\Re{s}>1/p$, where $1/p+1/q=1$. If $\mathcal{N}^{\perp}\cap \mathrm{BMOA}=\lbrace 0\rbrace$, then $\zeta(s)\neq 0$ for $\Re{s}\geq 1$.
\end{cor}
\begin{proof} Notice that $H^q$ is the Cauchy dual of $H^p$ and $q\geq 2$ implies that $1<p\leq 2$. In this range the $H^p$ are Banach spaces, and in particular Fr\'{e}chet spaces. Similarly the BMOA space is the Cauchy dual of $H^1$. Therefore the result follows by the converse in Proposition \ref{density <-> orthogonality} and by Theorem \ref{Noor density result}.
\end{proof}
The next result relates Toeplitz operators on $H^p$ and the orthogonality question. Recall that if $\varphi\in N^+$ is a unit, then $1/\varphi\in N^+$ and hence both $T_\varphi$ and its inverse $T_{1/\varphi}$ are densely defined Toeplitz operators on $H^p$ for $p>1$ (see Subsection 1.4).



\begin{prop} \label{hk/phi Hp}
Let $\varphi$ be a unit in $N^+$ with $T_\varphi$ the Toeplitz operator on $H^p$ for some $p>1$. If $T_\varphi \mathcal{N}=\varphi\mathcal{N}$ is dense in $H^p$, then \[\mathcal{N}^\perp \cap \mathrm{dom}_q(T_{1/\varphi}^*) = \{0\}, \ \ \  where  \ \frac{1}{p}+\frac{1}{q}=1.\]  The Hilbertian case $p=2$ gives $\mathcal{N}^\perp \cap \mathcal{H}(b)= \{0\}$ where $(b,a)$ is the Pythagorean pair associated with $1/\varphi\in N^+$.
\end{prop}
\begin{proof}

Let $g \in \mathcal{N}^\perp\cap \mathrm{dom}_q(T_{1/\varphi}^*)$. Since both $g$ and $h_k$ belong to $H^2$ for all $k\geq 2$, the Cauchy duality $\< h_k, g \>$ coincides with the $H^2$-inner product $\< h_k, g \>_2$. Hence
\[
\< \varphi h_k, T^*_{1/\varphi}g \> = \< h_k, g \> = \< h_k, g \>_2 = 0 \ \ \forall \ k\in\mathbb{N}
\]
 which implies that $T^*_{1/\varphi}g=0$ in $H^q$ by the density of $\varphi \mathcal{N}$ in $H^p$. Since $1/\varphi=b/a$ is a unit in $N^+$ (as the inverse of $\varphi$), both $a$ and $b$ are $H^\infty$ outer functions in $H^p$. So $T_{1/\varphi}(a H^p)=b H^p$ shows that $T_{1/\varphi}$ has dense range in $H^p$ since $b$ is outer. So $T^*_{1/\varphi}$ is injective and therefore $g=0$ in $H^q$. This concludes the general case. The Hilbertian case $p=2$ now follows by Theorem \ref{dom of adjoint}.
\end{proof}

We shall derive two non-trivial applications of this result. The first one extends Theorem \ref{Nperp local D} to all local Dirichlet spaces $\mathcal{D}_1^p$  with $p>1$. We note that the classical $\mathcal{D}_{\delta_1}$ is just $\mathcal{D}_1^2$ which is strictly smaller than $\mathcal{D}_{1}^p$ for $p\in (1,2)$ (see Subsection 1.2). We will need with the following approximation result from \cite{Aditya preprint}.

\begin{lem} \label{(I-S)N dense Hp} 
Let $\mu$ the M\"obius function. Then 
\[
\sum_{k=2}^n \frac{\mu(k)}{k} (I-S) h_k \to 1-z
\]
in the $H^p$ norm for all $p>0$.
\end{lem}
We observe that $I-S=T_\varphi$ where $\varphi(z)=1-z$ is an $H^\infty$ outer function since $\mathbb{C}[z]\varphi$ is dense in $H^2$. In particular $\varphi$ is a unit in $N^+$. Define operators on $H^p$ by
\begin{equation}\label{W_n}
     (W_n)f(z)=(1 + z + \dots + z^{n-1}) f(z^n) =\frac{1-z^n}{1-z}f(z^n)
\end{equation}
and $(T_n)f(z)=f(z^n)$ for $n\geq 1$ and $f\in H^p$. The  multiplicative semigroup of operators $(W_n)_{n\geq 1}$ was introduced in \cite{Noor} and is the main object of study in \cite{Juan preprint}. They are bounded on $H^p$ for $p>1$ (see \cite[Cor. 4.6]{Aditya preprint}). We shall need the identities 
\[
W_n \mathcal{N}\subset\mathcal{N} \ \ \mathrm{and} \ \ T_n(I-S)=(I-S)W_n
\]
for $k,n\geq 1$ which appear in \cite[p. 249]{Noor}. We are ready for the first application.

\begin{thm} \label{Nperp local D Hp}
We have
$
\mathcal{N}^\perp \cap \mathcal{D}_1^p = \{0\} 
$ for all $p>1$. 

\end{thm}
\begin{proof}
Let $\varphi(z) =1-z$. By Propositions \ref{dom D_zeta^p} and \ref{hk/phi Hp} we only need to prove that $\varphi\mathcal{N}$ is dense in $H^p$ for $p>1$. First note that Lemma \ref{(I-S)N dense Hp} implies that $\varphi$ belongs to the $H^p$-closure of $\varphi\mathcal{N}=(I-S)\mathcal{N}$. This in turn implies that $T_n\varphi$ belongs to the $H^p$-closure of $T_n(\varphi\mathcal{N})=\varphi W_n\mathcal{N}\subset\varphi\mathcal{N}$ for all $n\geq 1$. So in particular $\mathrm{span}(T_n\varphi)_{n\geq 1}\subset\ \mathrm{clos}_{H^p}(\varphi\mathcal{N})$. Now $\mathrm{span}(T_n\varphi)_{n\geq 1}=\mathrm{span}(1-z^n)_{n\geq 1}=\mathbb{C}[z]\varphi$ which is dense in $H^p$ for all $p>1$ because $\varphi$ is an $H^\infty$ outer function. This proves that $\mathrm{clos}_{H^p}(\varphi\mathcal{N})=H^p$ and concludes the proof.
\end{proof}
Our second application of Proposition \ref{hk/phi Hp} utilizes recent discoveries in $\mathcal{H}(b)$-space theory to obtain zero-free half-planes for $\zeta$. In view of Corollary \ref{Nperp interects H^p}, we would like to know when the $H^p$ and BMOA spaces are contained in some $\mathcal{H}(b)$ for $\varphi=b/a\in N^+$. Fortuitously for us, these problems were completely solved recently in a preprint by Malman and Seco \cite{Malman-Seco}. They show that $H^{\tilde{p}}\subset\mathcal{H}(b)$ for $\tilde{p}\in(2,\infty)$ if and only if $\varphi\in H^{p}$ where $p=\frac{2\tilde{p}}{\tilde{p}-2}\in(2,\infty)$, and also that $H^\infty\subset\mathrm{BMOA}\subset\mathcal{H}(b)$ if and only if $\varphi\in H^2$. By definition we always have $\mathcal{H}(b)\subset H^2$, and $\mathcal{H}(b)=H^2$ precisely when $\varphi\in H^\infty$. Therefore it makes sense to allow the values $p=2$ and $p=\infty$.
\begin{thm}\label{phiN dense zero free halfplane}
    Suppose $\varphi$ is a unit in $N^+$ such that $1/\varphi\in H^{p}$ for some $p\in(2,\infty)$. If $\varphi\mathcal{N}$ is dense in $H^2$, then $\zeta(s)\neq 0$ for $\Re{s}>\frac{1}{2}+\frac{1}{p}$. The case $p=2$ gives the Prime Number Theorem ($\Re{s}\geq 1$) and $p=\infty$ gives the RH ($\Re{s}>1/2$). 
\end{thm}
\begin{proof} By the results of Malman and Seco \cite{Malman-Seco} mentioned above, $1/\varphi$ belongs to $H^2$ or to $H^\infty$ precisely when $\mathcal{H}(b)$ contains BMOA or $\mathcal{H}(b)=H^2$ respectively, where $1/\varphi=b/a$ and $(b,a)$ the associated Pythagorean pair. Hence the cases $p=2,\infty$ follow by Corollary \ref{Nperp interects H^p} and Proposition \ref{hk/phi Hp}. For the case when $1/\varphi\in H^{p}$ for $p\in(2,\infty)$,  we have $H^{\tilde{p}}\subset\mathcal{H}(b)$ where $p=\frac{2\tilde{p}}{\tilde{p}-2}$ (again by Malman and Seco) or equivalently $\tilde{p}=\frac{2p}{p-2}$. If $1/\tilde{p}+1/\tilde{q}=1$, then we see that $\tilde{q}=\frac{2p}{(p+2)}$ and hence $1/\tilde{q}=1/2+1/p$. The result again follows by Corollary \ref{Nperp interects H^p} and Proposition \ref{hk/phi Hp}.
\end{proof}
An important distinction between Theorem \ref{Noor density result} and Theorem \ref{phiN dense zero free halfplane} is that in the former one must solve density problems in $H^p$ spaces that are non-Hilbertian, while in the latter the density problems are always in $H^2$. The following simple examples of $\varphi$ satisfy the hypothesis of  Theorem \ref{phiN dense zero free halfplane}. Let $\varphi(z)=(1-z)^\alpha$ for some $0<\alpha<1/2$. Then $\varphi$ is an $H^\infty$ outer function and hence a unit in $N^+$ with the property that $1/\varphi\in H^p$ for some $2<p<1/\alpha$. It follows that the density of $\varphi \mathcal{N}$ in $H^2$ would give a new zero-free half-plane for $\zeta$.

\section{The biorthogonality question}
Define the sequence of polynomials \(\{u_k: k \geq 2\}\) by
\begin{equation} \label{gk}
	u_k(z) = \sum_{d | k} \frac{\mu(k/d)}{k/d} (z^{d-1} - z^d), 
\end{equation}
where \(\mu\) denotes the Möbius function and $d | k$ denotes $d$  divides $k$. The main goal of this section is to prove the following theorem.

\begin{thm} \label{mainthm} \(\{u_k: k \geq 2\}\) is complete and biorthogonal to \(\{h_k: k \geq 2\}\) in \(H^2\).
\end{thm}
 Balazard \cite{balazard} noted that with the additional vector \(u_1(z) = 1-z\) the sequence \(\{u_k: k \geq 1\}\) is complete. However it is no longer minimal following Theorem \ref{mainthm}. We first make a key observation. Note that  \(u_k = (I - S)v_k\), where 
\begin{equation} \label{lk}
	v_k(z) =  \sum_{d | k} \frac{\mu(k/d)}{k/d} z^{d-1}.
\end{equation}
 It follows that $\la h_k,\, u_j\ar = \la (I - S^*) h_k, v_j\ar.$
Hence to show that \(\{u_k: k \geq 2\}\) and \(\{h_k: k \geq 2\}\) are biorthogonal, it is suffices to show that 
\[
\la (I - S^*) h_k, v_j\ar=\delta_{kj}.
\]

The proof of Theorem \ref{mainthm} will be divided into four steps. \\ \\
$\mathbf{Step \ 1.}$ \emph{Calculate the Fourier coefficients of \((I-S^*)h_k\).} \\ \\
$\mathbf{Step \ 2.}$ \emph{Prove \(\{u_k: k \geq 2\}\) is biorthogonal to \(\{h_k: k \geq 2\}\).}\\ \\
$\mathbf{Step \ 3.}$ \emph{Characterize all sequences biorthogonal to \(\{v_k: k \geq 2\}\) .} \\ \\
$\mathbf{Step \ 4.}$ \emph{Show that \(\{h_k: k \geq 2\}\)\ is the unique biorthogonal sequence for \(\{u_k: k \geq 2\}\).} \\ \\ 
The $\mathbf{Step \ 4}$ implies the completeness of $\{u_k: k \geq 2\}$ in $H^2$ by Proposition \ref{biorthogonality}. \\

$\mathbf{Step \ 1.}$  We first calculate the Fourier coefficients of \((I-S^*)h_k\).

\begin{lem} \label{coefhk}
	 We have
	\begin{equation*}
		(I-S^*)h_k(z) = \sum_{n=0}^\infty B_k(n+1) z^n
	\end{equation*}
 for all $k \geq 2$ where
	\begin{equation*}
		B_k(n) =
		\begin{cases}
			\frac{k}{n} - \frac{1}{n}, \quad k | n \\
			-\frac{1}{n}, \quad k \nd n
		\end{cases}.
	\end{equation*}
\end{lem}

\begin{proof}
	Note that if \(f(z) = \sum_{n=0}^\infty a_n z^n\), then
	\[
	S^*f(z) = \sum_{n=0}^\infty a_{n+1} z^n.
	\]
	
	Let \(c_n(k)\) be the Fourier coefficients of \(h_k\), i.e.,
	\[h_k(z) = \sum_{n=0}^\infty c_n(k) z^n.\]
	Then, the \(n\)-th Fourier coefficient of \((I-S^*)h_k\) is \(c_n(k) - c_{n+1}(k)\).
	The coefficients \(c_n(k)\) are calculated in \cite[p. 249]{Noor}:
	\begin{equation} \label{cnk}
		c_n(k) = H(n) - H\left(\frac{n}{k}\right) - \log k,
	\end{equation}
	where \(H(x):= \sum_{n \leq x} \frac1n\) for \(x > 0\) and \(H(0)=0\). It follows from (\ref{cnk}) that
	\[
		\begin{split}
			c_{n-1}(k) - c_n(k) &= H(n-1) - H\left(\frac{n-1}{k}\right) -\log k - \left(H(n) - H\left(\frac{n}{k}\right) -\log k\right) \\
			&= -\frac{1}{n} + \sum_{\frac{n-1}{k} < m \leq \frac{n}{k}} \frac{1}{m}.
		\end{split}
	\]
	Note that if there is some \(m \in \NN\) such that \(\frac{n-1}{k} < m \leq \frac{n}{k}\), then \(mk \leq n < mk + 1\), so that \(n = mk\). Therefore, the sum above is non-zero if and only if \(k | n\). Then,
	\begin{equation} \label{cd1d}
		B_k(n) = c_{n-1}(k) - c_n(k) =
		\begin{cases}
			-\frac{1}{n}, \quad k \nd n\\
			\frac{k}{n} - \frac{1}{n}, \quad k | n
		\end{cases}.
	\end{equation}
\end{proof} 

$\mathbf{Step \ 2.}$
We are now able to prove the first part of Theorem \ref{mainthm}.

\begin{thm} \label{minimal}
	\(\{u_k: k \geq 2\}\) is biorthogonal to \(\{h_k: k \geq 2\}\).
\end{thm}
\begin{proof}
	By $\mathbf{Step \ 1}$ it suffices to prove that
	\[ \sum_{d | j} B_k(d) \frac{\mu(j/d)}{j/d} = \la (I- S^*)h_k, v_j \ar = \delta_{kj}, \qquad \forall k,j \geq 2.\]
	There are two cases:
	\begin{itemize}
		\item[(i)] \(k \nd j\).
		Then, \(k \nd d\) for every \(d | j\), therefore
		\[
		\sum_{d | j} B_k(d) \frac{\mu(j/d)}{j/d} = \sum_{d | j} -\frac{1}{d} \frac{\mu(j/d)}{j/d} = -\frac{1}{j} \sum_{d | j} \mu(j/d) = -\frac{1}{j} \left\lfloor \frac{1}{j} \right\rfloor = 0,
		\]
		since \(j \geq 2\) and by the basic relation \(\sum_{d|k} \mu(d) = \lfloor 1/k \rfloor\).
		
		\item[(ii)]\(k | j\).
		Let \(q = \frac{j}{k}\). Then
		\begin{equation*}
			\sum_{d | j} B_k(d) \frac{\mu(j/d)}{j/d} = \sum_{\substack{d | j \\ k \nmid d}} -\frac{1}{d} \frac{\mu(j/d)}{j/d} + \sum_{\substack{d | j \\ k | d}} \left(\frac{k}{d} - \frac{1}{d}\right) \frac{\mu(j/d)}{j/d}. \\ 
		\end{equation*}
		\end{itemize} 
		The last sum is summing over those \(d\) that satisfy \(k | d | j\). However, \(k | d \iff d=mk\) for some \(m \in \NN\). Since \(j = qk\), it follows that \(d | j \iff m | q\). Hence, the last sum can be written as
		\[
		\sum_{\substack{d | j \\ k | d}} \left(\frac{k}{d} - \frac{1}{d}\right) \frac{\mu(j/d)}{j/d} = \sum_{\substack{d | j \\ k | d}} - \frac{1}{d}\frac{\mu(j/d)}{j/d} + \sum_{m | q} \frac{1}{m} \frac{\mu(q/m)}{q/m}.
		\]
		Therefore,
		\begin{align*}
			\sum_{d | j} B_k(d) \frac{\mu(j/d)}{j/d} &= \sum_{d|j} -\frac{1}{d} \frac{\mu(j/d)}{j/d} + \sum_{m|q}\frac{1}{m} \frac{\mu(q/m)}{q/m} \\
			&= -\frac{1}{j} \sum_{d|j} \mu(j/d) + \frac{1}{q} \sum_{m|q} \mu(q/m) = -\frac{1}{j} \left\lfloor \frac{1}{j} \right\rfloor + \frac{1}{q}\left\lfloor \frac{1}{q} \right\rfloor. \\
		\end{align*}
		Since \(j \geq 2\), the first term is always 0. On the other hand, the second term equals 1 if \(q=1\) and equals \(0\) otherwise. Finally note that \(q=1 \iff k=j\), and hence that
		$ \la h_k, u_j \ar = \delta_{kj}$ for all $k,j \geq 2$.
   \end{proof}


$\mathbf{Step \ 3.}$ We next characterize all sequences in $H^2$ biorthogonal to $\{v_k : k \geq 2\}$.

\begin{lem} \label{biortolk}
	A sequence \(\{f_k : k \geq 2\} \subset H^2\) is biorthogonal to \(\{v_k : k \geq 2\}\) if and only if there exists a sequence \((c_k)_{k\geq2} \in \CC^\NN\) such that
	\[
	f_k(z) = \sum_{n=0}^\infty A_{k}(n+1)\ z^n, \quad \forall k \geq 2,
	\]
	where the sequence $(A_k(n))_{n\geq 1}$ for each $k\geq2$ is defined by
	\[
	A_{k}(n) =
	\begin{cases}
		\frac{c_k}{n}+\frac{k}{n} , \quad k | n \\
		\frac{c_k}{n}, \ \ \ \ \   \quad k \nd n
	\end{cases}.
	\]
\end{lem}

\begin{proof}
	Let \(\{f_k : k \geq 2\} \subset H^2\) be a sequence biorthogonal to \(\{v_k : k \geq 2\}\) and \(A_k: \NN \to \CC\) be the arithmetical functions that satisfy
	\[ f_k(z) = \sum_{n=0}^\infty A_k(n+1) z^n, \quad \forall k \geq 2.\]
	Since the coefficients of \(v_j\) are real, the biorthogonality condition becomes
	\begin{equation} \label{lkc1}
		\sum_{d | j} A_k(d) \frac{\mu(j/d)}{j/d} = \la f_k, v_j \ar = \delta_{kj}, \quad \forall k,j \geq 2.
	\end{equation}
	Let \(I_k, \nu: \NN \to \CC\) be arithmetic functions defined by \(I_k(n) = \delta_{kn}\) and \(\nu(n) = \frac{\mu(n)}{n}\). Then (\ref{lkc1}) is equivalent to
	\begin{equation} \label{condak}
		\forall \ k \geq 2,\, \exists \ c_k \in \CC\ \text{ such that} \quad  A_k * \nu = c_kI_1 + I_k,
	\end{equation}
    where \(*\) denotes the Dirichlet product (see \cite[Section 2.6]{apostol}).
	Indeed, (\ref{lkc1}) doesn't impose any restriction on \(A_k * \nu (1)\), since it only need to hold for \(j \geq 2\), hence \(c_k = A_k * \nu(1)\) is free, so (\ref{lkc1}) and (\ref{condak}) are indeed equivalent. Notice that
	\[
	\sum_{d|k} \frac{\mu(k/d)}{k/d} \frac{1}{d} = \frac{1}{k} \left\lfloor \frac{1}{k} \right\rfloor = I_1(k),
	\]
	i.e., \(\nu^{-1}(n) = \frac{1}{n}\), since \(I_1\) is the unity with respect to $*$. Moreover
	\[
	I_k * \nu^{-1}(n) = \sum_{d | n} \delta_{kd} \frac{1}{n/d} =
	\begin{cases}
		\frac{k}{n}, \quad k|n \\
		0, \quad k \nd n
	\end{cases}.
	\]
	Therefore (\ref{condak}) is equivalent to the  statement that
	\begin{align*}
		\forall \ k \geq 2,\, \exists \ c_k \in \CC \text{ such that } A_k(n) &= c_k\nu^{-1}(n) + I_k * \nu^{-1}(n) \\
		&=
		\begin{cases}
			  \frac{c_k}{n}+\frac{k}{n}, \quad k|n \\
			\frac{c_k}{n},  \ \ \ \ \ \quad k \nd n
		\end{cases}.
	\end{align*}
	Hence the biorthogonality condition \eqref{lkc1} is equivalent to the condition above as desired. Finally \(f_k\in H^2\) since its coefficient sequence \(A_{k}\) clearly belongs to \(\ell^2\). 
\end{proof}





$\mathbf{Step \ 4.}$ In this final step we show that $\{u_k : k \geq 2\}$ is complete in $H^2$ by proving that $\{h_k : k \geq 2\}$ is uniquely biorthogonal to $\{u_k : k \geq 2\}$ in $H^2$ by Proposition \ref{biorthogonality}. To do so, recall that \(u_k = (I - S)v_k\) (see (\ref{lk})) implies
\[\la \phi_k,\, u_j\ar = \la (I - S^*) \phi_k, v_j\ar
\] 
for any sequence $\{\phi_k : k \geq 2\}$ in $H^2$. This implies that $I-S^*$ maps sequences biorthogonal to $\{u_k : k \geq 2\}$ onto sequences biorthogonal to $\{v_k : k \geq 2\}$ in the image of $I-S^*$. This correspondece is one-to-one since $I-S^*$ is injective on $H^2$. Therefore it is enough to prove that $((I-S^*)h_k)_{k\geq 2}$ is the unique sequence in the image of \(I-S^*\) that is biorthogonal to \(\{v_k: k \geq 2\}\).

\begin{lem} \label{ultimolema}
	A sequence \(\{f_k : k \geq 2\} \subset (I-S^*)H^2 \) is biorthogonal to \(\{v_k : k \geq 2\}\) if and only if
	\begin{equation} \label{fk}
		f_k(z) = \sum_{n=0}^\infty B_k(n+1) z^n = (I-S^*)h_k,
	\end{equation}
	where \(B_k\) are the sequences defined in Lemma~\ref{coefhk}.
\end{lem}

\begin{proof}
	Let \(\{f_k : k \geq 2\} \subset (I-S^*)H^2\) be a sequence biorthogonal to \(\{v_k : k \geq 2\}\) and let \(\fii_k \in H^2\) such that \(f_k = (I-S^*)\fii_k\). If \((b_k(n))_{n\geq 0}\) are the Maclaurin coefficients of \(\fii_k\), then
	\[
	f_k(z) = \sum_{n=0}^\infty (b_k(n) - b_k(n+1)) z^n.
	\]
	It then follows by Lemma~\ref{biortolk} that for each \(k \geq 2\), there exists a \(c_k \in \CC\) such that
	\[ 
	b_k(n-1) - b_k(n) = A_{k}(n) = 
	\begin{cases}
		\frac{k}{n} + \frac{c_k}{n}, \quad k|n \\
		\frac{c_k}{n}, \quad k \nd n
	\end{cases},
	\quad \forall n \geq 1.
	\]
	By induction, we obtain
	\begin{align*}
		b_k(n) &= b_k(0) - \sum_{j=1}^n A_{k}(j) = b_k(0) - \sum_{j\leq n} \frac{c_k}{j} - \sum_{\substack{j \leq n \\ k | j}} \frac{k}{j}\\
		&= b_k(0) - c_k \sum_{j \leq n} \frac{1}{j} - \sum_{m \leq n/k} \frac{1}{m} = b_k(0) - c_k H(n) - H\left(\frac{n}{k}\right),
	\end{align*}
	where \(H\) is the same function used in (\ref{cnk}). Since \(\fii_k \in H^2\), we get \((b_k(n))_n\in\ell^2\) and hence \(\lim_{n \to \infty} b_k(n) = 0\). So \((c_kH(n) + H(n/k))_n\) converges. Using Euler summation, one gets (see \cite{Noor}) 
	\begin{equation} \label{H}
		H(x) = \log x + \gamma + O\left(\frac{1}{x}\right),
	\end{equation}
	where \(\gamma\) is Euler-Mascheroni constant. Therefore,
	\begin{align*}
		c_kH(n) + H\left(\frac{n}{k}\right) &= c_k \log n + c_k\gamma + \log n - \log k + \gamma + O\left(\frac{k}{n}\right)  \\
		&= (c_k + 1) \log n + (c_k + 1) \gamma - \log k + O\left(\frac{k}{n}\right),
	\end{align*}
	which converges as \(n \to \infty\) if and only if \(c_k = -1\). Hence \(c_k = -1\) for all \(k \geq 2\). In that case \(A_{k} = B_k\) and we obtain (\ref{fk}). The converse is equivalent to Theorem~\ref{minimal} by the remarks at the start of Step $4$.
\end{proof}
As a consequence of Theorem \ref{Noor density result} and Proposition \ref{biorthogonality} the RH holds if and only if $(u_j)_{j\geq 2}$ is the unique sequence in \(H^2\) that is biorthogonal to $(h_k)_{k\geq2}$. On the other hand the next result shows what happens if some $\zeta$-zero violates the RH.
\begin{cor} If $\zeta(s_0)=0$ for some $1/2<\Re{s_0}<1$, then 
\[
\la h_k, u_j+\kappa_{s_0} \ar = \delta_{kj} \ \ \forall \ \ k,j \geq 2
\]
where $\kappa_{s_0}$ is the zeta kernel at $s_0$. So $(u_j+\kappa_{s_0})_{j\geq 2}$ is also biorthogonal to $(h_k)_{k\geq2}$.
\end{cor}
\begin{proof} This follows by \eqref{Zeta Kernel Relation} and Theorem \ref{minimal} since $\langle h_k,\kappa_{s_0}\rangle=0$ for all $ k\geq 2$.
\end{proof}
\section{The RH-failure conjecture}

The RH-failure (RHF) conjecture states that if the RH is false, then $\zeta(s)=0$ for infinitely many $s\in\mathbb{C}$ with $1/2<\Re{s}<1$. Our goal is to prove the following.

\begin{thm} \label{RHF implies dim N perp 0 or infty}
    The RHF conjecture implies that \(\dim (\mathcal{N}^{\perp})\) is either 0 or \(\infty\).
\end{thm}

Let \(\mathcal{K}:= \{\kappa_s: \Re{s}>1/2\}\) denote the family of zeta kernels. If $\zeta(s)=0$ for some $\Re{s}>1/2$, then $\langle h_k,\kappa_s\rangle=0$ for all $k\geq 2$ by \eqref{Zeta Kernel Relation} and hence $\kappa_s\in\mathcal{N}^\perp$ . So the RHF conjecture implies that $\mathcal{N}^\perp\cap\mathcal{K}$ is either empty (by Theorem \ref{Noor density result}) or has infinitely many elements. Therefore Theorem \ref{RHF implies dim N perp 0 or infty} follows if we show that $\mathcal{K}$ is linearly independent in $H^2$. We first show that elements of $\mathcal{K}$ are common eigenvectors for the adjoints of operators $(W_n)_{n\geq1}$ defined in \eqref{W_n}. For $f\in H^2 $ and $n\in\mathbb{N}$, we have
\[W^*_nf(z)=\sum_{k=0}^\infty [\hat{f}(nk)+\hat{f}(nk+1)+\ldots+\hat{f}(nk+n-1)]z^k \ \ 
\]
where $\hat{f}(n)$ denotes the $n$-th Fourier coefficient of $f$. This formula first appeared in \cite{Juan preprint}. It is possible to describe the common eigenvectors of $(W_n^*)_{n\geq1}$ completely.

\begin{prop}\label{common W^*_n eigenvectors}
    A non-zero \(f \in H^2\) is a common eigenvector for $(W_n^*)_{n\geq 1}$ if and only if there exists a multiplicative sequence $(\lambda_n)_{n\geq 1}$ with $(\lambda_{n+1}-\lambda_n)_{n\geq 1}\in\ell^2$ and 
    \begin{equation}\label{f hat}\hat{f}(n)=(\lambda_{n+1}-\lambda_n)\hat{f}(0) \ \ \forall \ \ n\geq 1.
    \end{equation}  
    Moreover \(W_n^*f=\lambda_nf\) for all $n\geq 1$. 
\end{prop}
By a multiplicative sequence $(\lambda_n)_{n\geq 1}$ we mean that $\lambda_n\lambda_m=\lambda_{nm}$ and $\lambda_1=1$. Similarly one can see that  $W_nW_m=W_{nm}$ and $W_1=I$ by \eqref{W_n}.
\begin{proof}
    Let $W_n^* f = \lambda_n f$ for $n\geq 1$ and some sequence $(\lambda_n)_{n\geq 1}$. Since $W_1^*=I$ and \(W_{nm}^* = W_n^* W_m^*\) it follows that $(\lambda_n)_{n\geq 1}$ is multiplicative. Furthermore
    \[
    \lambda_n \hat{f}(k) = \<W_n^* f, z^k\> = \<f, W_n z^k\> = \la f, \sum_{j=0}^{n-1} z^{nk+j} \ar = \sum_{j=0}^{n-1} \hat{f}(nk+j).
    \]
    which gives
    $\hat{f}(n)= \lambda_{n+1}\hat{f}(0) - \lambda_n \hat{f}(0)$ for all  $n\geq 1$ and hence $(\lambda_{n+1}-\lambda_n)_{n\geq 1}\in\ell^2$. Conversely suppose \(f\) is a non-zero function satisfying \eqref{f hat} for some multiplicative $(\lambda_n)_{n\geq 1}$ with $(\lambda_{n+1}-\lambda_n)_{n\geq 1}\in\ell^2$. Normalizing by supposing \(\hat{f}(0) = 1\), we get
    \begin{align*}
        (W_n^*f)(z) &= \sum_{k=0}^\infty \left(\sum_{j=0}^{n-1} \hat{f}(nk+j) \right) z^k = \sum_{j=0}^{n-1} \hat{f}(j) + \sum_{k=1}^\infty \left(\sum_{j=0}^{n-1} \hat{f}(nk+j) \right) z^k \\
        &= \lambda_n + \sum_{k=1}^\infty (\lambda_{nk+n} - \lambda_{nk})z^k = \lambda_n \left(1+\sum_{k=1}^\infty (\lambda_{k+1} - \lambda_k) z^k \right)= \lambda_n f(z)
    \end{align*}
   for all $n \geq 2$. So \(f \in H^2\) is a common eigenvector for $(W_n^*)_{n\geq 1}$.
\end{proof}

Choosing $\lambda_k=k^{1-\bar{s}}$ and $\hat{f}(0)=-1/\bar{s}$ in Proposition \ref{common W^*_n eigenvectors} for any fixed $\Re{s}>1/2$ shows that each $\kappa_s\in\mathcal{K}$ is a common eigenvector for $(W_n^*)_{n\geq 1}$ (see \eqref{kappa_s}) with
\begin{equation}\label{kapp_s eigenvalue}
W_n^*\kappa_s=n^{1-\bar{s}}f \ \ \forall \ n\geq 1.
\end{equation}
We want to prove that for any finite subset $\{\kappa_{s_1},\ldots,\kappa_{s_\ell}\}\subset\mathcal{K}$ there exists some $W_n^*$ such that the corresponding eigenvalues are all distinct. This will give us the linear independence of every finite subset of $\mathcal{K}$ and hence of 
$\mathcal{K}$ itself. First suppose that the real parts of \(s_1,\, \dots , s_\ell \) are all distinct. Since $|n^{1-\bar{s}}| = n^{1-\Re{s}}$ it follows that the eigenvalues of $W_n^*$ (for all $n>1$) corresponding to $\kappa_{s_1},\ldots,\kappa_{s_\ell}$ are all distinct. If the real parts of \(s_1,\, \dots , s_\ell \) are \emph{not} all distinct, then we need the following result.

\begin{lem}\label{Pigeon lemma}
    Given distinct \(a_1,\, \dots,\, a_n \in \RR\), at most finitely many primes \(p\) have the property that there exists a pair $a_i,a_j$ with $1 \leq i < j\leq n$ such that 
    \begin{equation} \label{eq1}
      \ (a_i - a_j) \log p \in 2\pi \ZZ.
    \end{equation}
   
    \end{lem}

\begin{proof}
    Suppose there are infinitely many primes that satisfy (\ref{eq1}). For each such prime \(p\) there exists some \(1 \leq i < j \leq n\) and \(k \in \ZZ \setminus \{0\}\) such that
    \[
    (a_i - a_j) \log p = 2\pi k \ \implies \ \frac{2 \pi k}{\log p} = a_i - a_j.
    \]
    But since there are only finitely many numbers \(a_i - a_j\) with \(i<j\), and none of which equal 0, there must exist distinct primes \(p,q\)  and \(k_1, k_2 \in \ZZ\setminus \{0\}\) such that
    \[
    \frac{2\pi k_1}{\log p} = a_i - a_j = \frac{2\pi k_2}{\log q} \ \implies \ k_2\log{p}=k_1\log{q}\neq 0.
    \]
    for some pair $i<j$. In particular, \(p^{k_2} = q^{k_1} \neq 1\), which is a contradiction.
\end{proof}

The following result then completes the proof of Theorem~\ref{RHF implies dim N perp 0 or infty}.

\begin{prop}
    The family of zeta kernels \(\mathcal{K}\) is linearly independent.
\end{prop}

\begin{proof} Let $\{\kappa_{s_1},\ldots,\kappa_{s_\ell}\}\subset\mathcal{K}$ be a finite subset. The case when the real parts of \(s_1,\, \dots , s_\ell \) are all distinct was already dealt with. Suppose some of the \(s_1,\, \dots , s_\ell \) have the same real parts. So \(\{s_1,\, \dots , s_\ell\}\) is the finite disjoint union of sets of the form $A_r:=\{s_i:\Re{s_i}=r, \ i=1,\ldots,\ell\}$ for $r\in\mathbb{R}$. It is enough to prove that the family $\{\kappa_s:s\in A_r\}$ is linearly independent when $A_r$ has more than one element. Since \(s_1,\, \dots , s_\ell \) are distinct complex numbers, the imaginary parts of elements in $A_r$, which we denote by $a_1,\ldots,a_n$, must all be distinct. Applying Lemma \ref{Pigeon lemma} to $a_1,\ldots,a_n$ shows that there exist infinitely many primes $q$ such that
\begin{equation} \label{condicao}
(a_i - a_j) \log q \notin 2\pi \ZZ, \quad \forall \ 1 \leq i<j \leq n. 
\end{equation}
For such a prime $q$, we claim that the $\{\kappa_s:s\in A_r\}$ are \(W_q^*\)-eigenvectors with distinct eigenvalues. To see this first note that \(W_q^*\kappa_s = q^{1-\bar{s}} \kappa_{s}\) by \eqref{kapp_s eigenvalue} and
\[
q^{1-\bar{s}}=e^{(1-\bar{s})\log{q}}=e^{(1-r)\log q}e^{i\mathrm{Im}{(s)}\log q} \ \ \forall \ s\in A_r.
\]
But Im$(s)$ for $s\in A_r$ are precisely the real numbers $a_1,\ldots,a_n$. Therefore the eigenvalues $q^{1-\bar{s}}$ for $s\in A_r$ are all distinct by \eqref{condicao} and hence $\{\kappa_s:s\in A_r\}$ and therefore all of $\mathcal{K}$ is linearly independent.  
\end{proof}

\section{Appendix}

Denote by $\mathbb{C}_\rho$ the half-plane $\{s\in\mathbb{C}:\Re{s}>\rho\}$. In this appendix we provide an alternate proof for the fundamental relation 
\begin{equation}\label{zeta relation 2}
\langle h_k, \kappa_{s}\rangle= -\frac{\zeta(s)}{s} (k^{1-s} - 1) \ \ \forall  \ s\in\mathbb{C}_{1/2}, \ k \geq 2.
\end{equation}
We first prove that (\ref{zeta relation 2}) holds for all $s\in\mathbb{C}_1$. We then prove that the function \(s \longmapsto \langle h_k,\kappa_s\rangle\) has an analytic continuation to $\mathbb{C}_0$ for each \(k \geq 2\). Since the right side of (\ref{zeta relation 2}) is already analytic for \(s\in\CC \setminus \{0\}\), the result then follows by analytic continuation. Recall from Subsection 1.7 that
\[
\kappa_s(z) = \sum_{n=0}^\infty \phi_n(\bar{s}) z^n \ \ \mathrm{where} \ \  \phi_n(s)=-\frac{1}{s}\left((n+1)^{1-s}-n^{1-s}\right).
\]

\begin{lem} \label{calcgeq1}
	The identity \eqref{zeta relation 2} holds for $s\in\mathbb{C}_1$.
	\end{lem}
\begin{proof} Let \((c_n(k))_n\) be the Fourier coefficients of \(h_k\). Since \(\overline{\phi_n(\overline{s})} = \phi_n(s)\), we have
	\begin{align*}
		\langle h_k, \kappa_{s}\rangle &= \sum_{n=0}^\infty c_n(k)\overline{\phi_n(\overline{s})} = \sum_{n=0}^\infty c_n(k) \phi_n(s) \\
    &= \lim_{N \to \infty} \left( -\frac{c_0(k)}{s} - \frac{1}{s} \sum_{n=1}^N c_n(k)\left((n+1)^{1-s} - n^{1-s}\right) \right) \\
		&= \lim_{N \to \infty} \left( -\frac{1}{s} \sum_{n=0}^N c_n(k) (n+1)^{1-s} + \frac{1}{s} \sum_{n=1}^N c_n(k)n^{1-s} \right) \\
		&= \lim_{N \to \infty} \left( -\frac{1}{s} \sum_{n=1}^N (c_{n-1}(k) - c_n(k)) n^{1-s} - \frac{1}{s} c_N(k)(N+1)^{1-s} \right).
	\end{align*}
	Since \(c_n(k) = O(k/n)\) (see \cite[p. 249]{Noor}), we have \(c_N(k)(N+1) = O(1)\). Furthermore \((N+1)^{-s} \to 0\) for \(\Re(s)>0\) and therefore we get
	
	\begin{align*}
		\langle h_k, \kappa_{s}\rangle &= -\frac{1}{s} \lim_{N \to \infty} \left( \sum_{n=1}^N (c_{n-1}(k) - c_n(k)) n^{1-s} \right) \\
  &\overset{(\ref{cd1d})}{=} -\frac{1}{s} \lim_{N \to \infty} \left( \sum_{n=1}^N -\frac{1}{n} n^{1-s} + \sum_{\substack{n=1 \\ k | n}}^N \frac{k}{n} n^{1-s} \right) \\
		&= -\frac{1}{s} \lim_{N \to \infty} \left( \sum_{n=1}^N n^{-s} + \sum_{m=1}^{\lfloor \frac{N}{k}\rfloor} \frac{1}{m} (mk)^{1-s} \right) \\
  &= -\frac{1}{s} \lim_{N \to \infty} \left( - \sum_{n=1}^N n^{-s} + k^{1-s} \sum_{m=1}^{\lfloor \frac{N}{k}\rfloor} \frac{1}{m} m^{1-s} \right) \\
		&\overset{(*)}{=} -\frac{1}{s} (-\zeta(s)) - \frac{k^{1-s}}{s} \zeta(s) = -\frac{\zeta(s)}{s} (k^{1-s} - 1), 
	\end{align*}
	where in \((*)\) we split the limit in two and use the definition of \(\zeta\) for \(\Re(s)>1\).
\end{proof}
The inner product $\langle h_k, \kappa_{s}\rangle$ defined for $s\in\CC_{1/2}$ also makes sense for $s\in\CC_0$.
\begin{lem} \label{hol}
	The function \(\Phi_k: \CC_{1/2} \to \CC\) defined by
	\begin{equation} \label{fik}
		\Phi_k(s) := \langle h_k, \kappa_{s}\rangle  = \sum_{n=0}^\infty c_n(k)\phi_n(s).
	\end{equation}
	 has an analytic continuation to \(\CC_0\) for each \(k \geq 2\).
\end{lem}
\begin{proof}
	Since each $\phi_n$ is holomorphic in $\mathbb{C}_0$, it is sufficient to prove that the series in (\ref{fik}) converges uniformly in every half-plane \(\CC_{\rho}\) for $\rho>0$. Note that 
 \[
 |\phi_n(s)|=\frac{|1-s|}{|s|}\left|\int_n^{n+1}y^{-s}dy\right|\leq\frac{|1-s|}{|s|}n^{-\Re{s}}=O(n^{-\rho})
 \]
 for $s\in\mathbb{C}_\rho$ with $\rho>0$. Also \(c_n(k) = O(k/n)\) for each $k\geq 2$, and hence we get $c_n(k)\phi_n(s)=O(n^{-1-\rho})$ for $s\in\mathbb{C}_\rho$. So
	$\Phi_k$ converges uniformly in \(\CC_\rho\) for $\rho>0$.    
\end{proof}

\section*{Acknowledgement}
This work was financed in part by the Coordenação de Aperfeiçoamento de Nivel Superior- Brasil (CAPES)- Finance Code 001. 
The fourth author is supported by grant \#2023 Provost of Inclusion and Belonging, University of São Paulo (USP).

\bibliographystyle{amsplain}

\end{document}